\documentclass[11pt,a4paper]{article}

\usepackage{amsmath, amsfonts, amssymb}
\usepackage{theorem}
\usepackage[dvips]{epsfig}

\usepackage{graphicx}

\usepackage[latin1]{inputenc}
\newtheorem{thm}{Theorem}[section]

\newtheorem{prop}[thm]{Proposition}

\newtheorem{cor}[thm]{Corollary}

\def\be#1\ee{\begin{equation}#1\end{equation}}
\newcommand{\bea}{\begin{eqnarray}}
\newcommand{\eea}{\end{eqnarray}}
\newcommand{\beaa}{\begin{eqnarray*}}
\newcommand{\eeaa}{\end{eqnarray*}}
\newcommand{\bei}{\begin{itemize}}
\newcommand{\eei}{\end{itemize}}
\newcommand{\bee}{\begin{enumerate}}
\newcommand{\eee}{\end{enumerate}}

\def\P{{\mathbb{P}}}

\def\R{\mathbb{R}}
\def\E{\mathbb{E}\,}

\newcommand{\eps}{\varepsilon}
\newenvironment{proof}[1][] {\noindent {\bf Proof#1:} }{\hspace*{\fill}$\square$\medskip\par}

\def\toalsur{\stackrel{\textrm{a.s.}}{\longrightarrow}}
\def\toprob{\stackrel{\P}{\to}}
\def\tomean{\stackrel{L_q}{\longrightarrow}}

\def\C{{\mathcal C}}
\def\CC{{\mathbb C}}
\def\Erfc{\textrm{Erfc}}

\def\V{{\mathbb V}}
\def\tW{\widetilde W}

\def \=L{\ {\buildrel\hbox{\scriptsize d }\over =}\ }
\def\iw{I_W}
\def\iwz{I_W^0}
\def\itwz{I_{{\tW_T}}^0}

\begin{document}

\title{\bf Energy of Taut Strings Accompanying Wiener Process}
\author{
   Mikhail Lifshits\footnote{St.Petersburg State University, Russia, Stary 
   Peterhof, Bibliotechnaya pl.,2, 
   email {\tt mikhail@lifshits.org} and MAI, Link\"oping University.} 
   \and Eric Setterqvist\footnote{MAI, Link\"oping University, 58183 
   Link\"oping, Sweden,  email \ {\tt eric.setterqvist@liu.se}.  
}}
\date{\today} 

\maketitle

\begin{abstract}
Let $W$ be a Wiener process. For $r>0$ and $T>0$ let
$\iw(T,r)^2$ denote the minimal value of the energy
$\int_0^T h'(t)^2 dt$ 
taken among all absolutely continuous functions $h(\cdot)$ 
defined on $[0,T]$, starting at zero and satisfying
\[
   W(t)-r \le h(t)\le W(t)+r,\qquad 0\le t \le T.
\]    
The function minimizing energy is a taut string, a classical object
well known in Variational Calculus, in Mathematical Statistics, and in a broad 
range of applications. 
We show that there exists a constant $\C\in (0,\infty)$ such that for any $q>0$
\[
    \frac{r} { T^{1/2}}\,  \iw(T,r) \tomean \C, 
    \qquad \textrm{as } \frac{r}{T^{1/2}}\to 0, 
\]   
and for any fixed $r>0$, 
\[
   \frac{r} { T^{1/2}}\,  \iw(T,r) \toalsur \C, 
   \qquad \textrm{as } T\to\infty.
\]
Although precise value of $\C$ remains unknown, we give various theoretical
bounds for it, as well as rather precise results of computer simulation. 

While the taut string clearly depends on entire trajectory of $W$, we also 
consider an adaptive version of the problem by giving a construction
(called Markovian pursuit) of a random function $h(t)$ based only on the values 
$W(s),s\le t$, and having minimal asymptotic energy. The solution, i.e. an optimal
pursuit strategy, turns out to be related with a classical minimization problem 
for Fisher information on the bounded interval. 
\end{abstract}
\vskip 1cm

\bigskip

\vfill

\newpage
\section*{Introduction}
\setcounter{equation}{0}

Given a time interval $[0,T]$ and two functional boundaries
$g_1(t)<g_2(t)$, $0\le t \le T$, the {\it taut string} is a function $h_*$
that for any (!) convex function $\varphi$ provides minimum for 
the functional
\[ F_\varphi(h):=\int_0^T \varphi(h'(t)) \, dt
\]
among all absolutely continuous functions $h$ with given starting 
and final values and satisfying
\[
   g_1(t) \le h(t)\le g_2(t),\qquad 0\le t \le T.
\]  
The list of simultaneously optimized functionals includes energy 
$\int_0^T h'(t)^2 dt$,
variation $\int_0^T |h'(t)| dt$, graph length 
$\int_0^T \sqrt{1+h'(t)^2} dt$, etc.    

The first instance of taut strings that we have found in the literature is in 
G. Dantzig's paper \cite{Dantzig}. Dantzig notes there that the problem 
under study and its solution was discussed in R. Bellman's seminar at RAND 
Corporation in 1952. The taut strings were later used in Statistics, see  
\cite{Mammen} and \cite{Davies}. 
In the book \cite[Chapter 4, Subsection 4.4]{Scherzer}, 
taut strings are considered in connection with problems in image processing. 
Quite recently, taut strings were applied to a buffer management 
problem in communication theory, see \cite{Setterqvist}.

In this article, we study the energy of the taut string going through the tube
of constant width constructed around sample path of a Wiener process $W$, 
i.e. for some $r>0$ we let $g_1(t):=W(t)-r$, $g_2(t):=W(t)+r$,
see Fig. \ref{fig:ts}.

\begin{center}
\begin{figure} [ht]

\begin{center}
\includegraphics[height=1.8in,width= 3.5in]{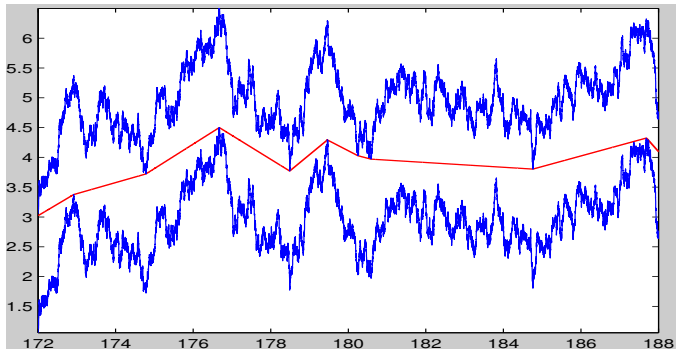}
\end{center}

\caption{A fragment of taut string accompanying Wiener process.}
\label{fig:ts}
\end{figure}
\end{center}


We focus attention
on the behavior in a long run: we show that when $T\to \infty$, the taut string
spends asymptotically constant amount of energy $\C^2$ per unit of time.
Precise assertions are given in Theorems \ref{t:mean} and \ref{t:alsur} below.
The constant $\C$ shows {\it how much energy an absolutely continuous 
function must spend if it is bounded to stay within a certain distance from the 
non-differentiable trajectory of $W$}. 

Although precise value of $\C$ remains unknown, we give various theoretical
bounds for it in Section \ref{s:estimates}, as well as the results 
of computer simulation in Section \ref{s:simulation}. The latter suggest 
$\C\approx 0.63$.

If we take the pursuit point of view, considering $h(\cdot)$ as a trajectory
of a particle moving with finite speed and trying to stay close to a Brownian
particle, then it is much more natural to consider constructions that define
$h(t)$ in adaptive way, i.e. on the base of the known $W(s),s\le t$. Recall
that the taut string depends on the entire trajectory  $W(s),s\le T$, hence
it does not fit the adaptive setting. In view of Markov property of $W$, the 
reasonable pursuit strategy for $h(t)$ is to move towards $W(t)$ with the speed 
depending on the distance $|h(t)-W(t)|$. In this class of algorithms we find
an optimal one in Section \ref{s:pursuit}. The corresponding function spends 
in average $\tfrac{\pi}{2}\approx 1.57$ units of energy per 
unit of time. Comparing of two constants shows that we have to pay more than
double price for not knowing the future of the trajectory of $W$.
To our great surprise, the search of optimal pursuit strategy boils down
to the well known variational problem: minimize Fisher information on the class
of distributions supported on a fixed bounded interval.
 
We conjecture that the provided algorithm is the optimal one in the entire 
class of adaptive algorithms.  
\medskip

In Section \ref{s:misc} we establish some connections with other well known
settings and problems.

First, we recall that the famous Strassen's functional law of the iterated logarithm
(FLIL) and its extensions handling convergence rates in FLIL actually deal exactly 
with the energy of taut strings.  Not surprisingly, we borrowed some techniques
for evaluation of this energy from FLIL research. Yet it should be noticed that
FLIL requires very different range of parameters $r$ and $T$ than those emerging
in our case. The FLIL tubes are much wider, hence the taut string energy is much
lower than ours.  This is why Strassen law with its super-slow loglog rates is so 
hard to reproduce in simulations, while our results handling the same type of 
quantities are easily observable in computer experiment. 

Second, we briefly look at the taut string as a minimizer of
of variation
\[
   \V(h):= \int_0^T |h'(t)| dt.
\]  
Since $|\cdot|$ is not a {\it strictly} 
convex function, the corresponding variational problem typically has 
{\it many} solutions.  In \cite{LoMil,Mil}  another minimizer of $\V(h)$ is 
described in detail, a so called {\it lazy function}. As E.~Schertzer pointed 
to us, the relations between the taut strings and lazy functions are yet to be 
clarified.

Finally, we briefly describe a discrete analogue of our problem thus giving 
flavor of eventual applications.
\medskip

As a conclusion, Section \ref{s:concl} traces some forthcoming or possible
developments of the treated subject.

\section{Notation and main results}
\label{s:nota}
\setcounter{equation}{0}

Throughout the paper, we consider uniform norms
\[
   ||h||_T:= \sup_{0\le t\le T} |h(t)|, \qquad h\in \CC[0,T],
\]
and Sobolev-type norms
\[
   |h|^2_T:= \int_0^T h'(t)^2 dt, \qquad h\in AC[0,T],
\]
where $AC[0,T]$ denotes the space of absolutely continuous functions on
$[0,T]$. It is natural to call  $|h|^2_T$ {\it energy}.

Let $W$ be a Wiener process. We are mostly interested in its approximation 
characteristics
\[
  \iw(T,r):=\inf\{|h|_T; h\in AC[0,T], ||h-W||_T\le r, h(0)=0  \}
\]
and
\[
  \iwz(T,r):=\inf\{|h|_T; h\in AC[0,T], ||h-W||_T\le r, h(0)=0, h(T)=W(T)  \}.
\]
The unique functions at which the infima are attained are called {\it taut string},
resp. {\it taut string with fixed end}.

Our main results are as follows. 

\begin{thm} \label{t:mean}
There exists a constant $\C\in (0,\infty)$ such that
if $\tfrac{r}{\sqrt{T}}\to 0$, then
\begin{eqnarray} 
   \label{toc}
   \frac{r} { T^{1/2}}\,  \iw(T,r) \tomean \C,
   \\
   \label{tocz}
   \frac{r} { T^{1/2}}\, \iwz(T,r) \tomean \C,    
\end{eqnarray}
for any $q>0$.
\end{thm}

We may complete the mean convergence with  almost sure convergence
to $\C$.

\begin{thm} \label{t:alsur}
For any fixed $r>0$, when $T\to\infty$, we have
\begin{eqnarray*}  
   \frac{r} { T^{1/2}}\,  \iw(T,r) \toalsur \C,
   \\
   \frac{r} { T^{1/2}}\, \iwz(T,r) \toalsur \C.   
\end{eqnarray*}
\end{thm}

\section{Basic properties of $\iw$ and $\iwz$}
\setcounter{equation}{0}
\label{s:basic}

We prepare the proofs of the main results given below in Subsections \ref{ss:asym2}
and \ref{ss:asym3}
by exploring scaling and concetration properties of the taut string's energy.

\subsection{Scaling}

Given two functions $W(t)$ and $h(t)$ on $[0,T]$, let us rescale them onto the 
time interval $[0,1]$ by letting
\[
  X(s):= \frac{W(sT)}{\sqrt{T}}, \ g(s):= \frac{h(sT)}{\sqrt{T}}, \qquad 0\le s\le 1. 
\] 
Then 
\[
   ||g-X||_1=\frac {||h-W||_T}{\sqrt{T}}
\]
and
\[
   |g|_1^2= \int_0^1 g'(s)^2 ds =  \int_0^1 \left( \frac{ h'(sT)T}{\sqrt{T}}\right)^2  ds 
   = T \int_0^1 h'(sT)^2  ds 
   =  \int_0^T h'(t)^2  dt
   =|h|^2_T.
\]
The boundary conditions are also transformed properly: namely, $h(0)=W(0)$ is equivalent to
$g(0)=X(0)$, while $h(T)=W(T)$ is equivalent to $g(1)=X(1)$. Therefore, $h$
belongs to the set $\{h: h\in AC[0,T], ||h-W||_T\le r, h(0)=0  \}$ iff 
$g$ belongs to the analogous set 
$\{g: g \in AC[0,1], ||g-X||_1\le \tfrac{r}{\sqrt{T}}, g(0)=0  \}$. 

Recall that if $W$ is a Wiener process on $[0,T]$, then $X(s):= \tfrac{W(sT)}{\sqrt{T}}$
is a Wiener process on $[0,1]$. We conclude that
\[
  \iw(T,r) \=L \iw  \left(1, \tfrac{r}{\sqrt{T}}\right). 
\] 
Similarly,
\be \label{scal0}
  \iwz(T,r) \=L  \iwz \left(1, \tfrac{r}{\sqrt{T}}\right). 
\ee 
Therefore, assertions \eqref{toc} and \eqref{tocz} may be rewritten in a one-parameter form
\begin{eqnarray} 
   \label{toc1}
   \eps \,  \iw(1,\eps) \tomean \C, \qquad \textrm{as } \eps\to 0,
   \\
   \nonumber
   \eps\, \iwz(T,\eps) \tomean \C, \qquad \textrm{as } \eps\to 0.
\end{eqnarray}

\subsection{Finite moments}

We will show now that both $\iw(T,r)$ and $\iwz(T,r)$ have finite exponential moments.
Yet in the following we only need that
\begin{eqnarray}
  D(T,R)  :=  \E \iw(T,r)^2 <\infty,
  \\
  D^0(T,R):=  \E \iwz(T,r)^2 <\infty.
\end{eqnarray}  
Let $v$ be an even integer. Then $\delta:=\tfrac 2v$ is inverse to an integer, and
we may cut the time interval $[0,1]$ into $\delta^{-1}$ intervals of length $\delta$.
Let $W_\delta$ be the linear interpolation of $W$ based on the knots $(j\delta,W(j\delta))$,
$0\le j\le \delta^{-1}$. Clearly, we have either $||W_\delta-W||_1>r$ or
$\iwz(1,r)\le |W_\delta|_1$. It follows that
\be \label{p2}
 \P\left(\iwz(1,r)^2>v\right)
 \le
  \P\left(||W_\delta-W||_1>r \right) +  \P\left(|W_\delta|_1^2>v \right). 
\ee
Notice that
\begin{eqnarray}
  \nonumber 
  ||W_\delta-W||_1 &=& \max_{0\le t\le 1} |W_\delta(t)-W(t)|
  \\ \nonumber
  &=&
  \max_{0\le j < \delta^{-1}} \ \max_{j\delta\le t\le (j+1)\delta} |W_\delta(t)-W(t)|
  \\ \label{sum1}
  &\=L& \max_{0\le j < \delta^{-1}} \sqrt{\delta} \max_{0\le t\le 1} |B_j(t)|,
\end{eqnarray}
where $(B_j)$ are independent Brownian bridges, and
\begin{eqnarray}
    \nonumber
    |W_\delta|_1^2 
    &=& \int_0^1 W_\delta'(t)^2 dt 
    \\ \label{sum2}
    &=& \delta^{-1} \sum_{0\le j < \delta^{-1}}
        \left(W((j+1) \delta)-W(j \delta)\right)^2 
    = \sum_{0\le j < \delta^{-1}} \eta_j^2, 
\end{eqnarray}
where $(\eta_j)$ are i.i.d. standard normal random variables.

Now we may evaluate the probabilities in \eqref{p2}. 
By using \eqref{sum1}, we obtain
\begin{eqnarray*}
    \P\left(||W_\delta-W||_1>r \right)
    &\le& 
    \delta^{-1} \ \P\left(||B||_1> \frac{r}{\delta} \right)
    \le 
    \delta^{-1} \ \P\left(||W||_1> \frac{r}{\delta} \right) 
    \\
    &\le& 
    \frac{v}{2}\cdot 2 \cdot \exp\left( - \frac{r^2}{2 \delta} \right)  
    =  v\ \exp (-r^2 v/4).  
\end{eqnarray*}
On the other hand, by using Cram\'er--Chernoff theorem and \eqref{sum2},
\[
  \P\left(|W_\delta|_1^2>v \right)
  = \P\left( \sum_{0\le j < v/2} \eta_j^2 >v \right)
  \le \exp\{-c_1 v\}
\]
for all $v$ and some universal constant $c_1$.
It follows that
\[
    \P\left(\iwz(1,r)^2>v\right)
    \le   v\ \exp (-r^2 v/4) + \exp\{-c_1 v\}.
\]
Hence,
\[
   \E \exp \left(c \, \iwz(1,r)^2\right) <\infty,
\]
whenever $0<c<\min\{\tfrac{r^2}{4}, c_1\}$. By scaling we also have
\[
   \E \exp \left(c\, \iwz(T,r)^2\right) <\infty,
\]
for any $r,T>0$ and sufficiently small positive $c$.
It follows from the definitions that
\be \label{comp_simple}
  \iw(T, r)\le  \iwz(T, r) \qquad \forall\, T,r>0.
\ee 
Hence, the exponential moment of $\iw(T,r)$ is finite, too.

\subsection{ Relations between $\iw$ and $\iwz$ }

We already noticed in \eqref{comp_simple} that
$\iw(T, r)\le  \iwz(T, r)$.
We will show now that a kind of converse estimate is also true.

\begin{prop} For all positive $T,r,\delta$ it is true that
\be \label{lowerE}
  \E \iw(T,r)^2 \ge  \E \iwz(T+1,r+\delta)^2 - \E \iwz(1,\delta)^2 -r^2.
\ee
\end{prop}

\begin{proof} 
Let us fix for a while the time interval $[0,1]$ and let
us approximate the trajectory of Wiener process $W$ by functions
starting from some arbitrary point $\rho\in\R$. Let $\delta>0$ and 
let $h(\cdot)$ be the taut string with fixed end at which $\iwz(1,\delta)$ 
is attained.
Then we have $h(0)=0$, $h(1)=W(1)$, $||h-W||_1\le \delta$, 
 $|h|_1=\iwz(1,\delta)$. Let
 \[
   H(t):= \rho+h(t)-\rho t, \qquad 0\le t\le 1.
 \]
 Then $H(0)=\rho+h(0)=\rho$,
 $H(1)=\rho+h(1)-\rho=h(1)=W(1)$,
 \be \label{HW}
     ||H-W||_1\le ||h-W||_1 +|\rho| \max_{0\le t\le 1}|1-t| \le \delta+|\rho|, 
 \ee
 and
 \begin{eqnarray} \nonumber
  |H|_1^2&=& \int_0^1 H'(t)^2 dt = \int_0^1(h'(t)-\rho)^2 dt
  \\ \nonumber
  &=& 
  \int_0^1 h'(t)^2 dt +\rho^2 -2\rho \int_0^1 h'(t) dt
  \\ \label{H1}
  &=& 
   |h|_1^2  +\rho^2 -2\rho (h(1)-h(0))
   = \iwz(1,\delta)^2 +\rho^2 -2\rho  W(1).
 \end{eqnarray} 

Now we pass to the lower bound for $\iw(T, r)$. Let us fix
$r,\delta,T$ and produce an approximation for $W$ on
$[0,T+1]$ with the fixed end. First, let $\widehat h (t), 0\le t\le T$, 
be the taut string at which $\iw(T,r)$ is attained. The end point
is not fixed, thus $\rho:=\widehat h(T)-W(T)$ need not vanish.
Nevertheless we still have
\[
  |\rho|\le ||\widehat h-W||_T\le r.
\]
Now we approximate the auxiliary Wiener process
\[
   \tW_T(s):=W(T+s)-W(T), \qquad 0\le s\le 1,
\]
by the function $H(\cdot)$ defined above and let
\[
  \widehat h(T+s):= W(T) + H(s), \qquad 0\le s\le 1. 
\]
At the boundary point $T$ the first definition yields the value
$\widehat h(T)=W(T)+\rho$, the second definition yields
$\widehat h(T):= W(T) + H(0)$; the two values coincide by the definition of 
function $H$.

Moreover,
\[
  \widehat h(T+1)= W(T) + H(1)= W(T)+ \tW_T(1)= W(T)+W(T+1)-W(T)=W(T+1). 
\]
Therefore, the extended function $\widehat h(\cdot)$ provides an absolutely
continuous approximation with fixed end to $W$ on $[0,T+1]$.
Furthermore, by \eqref{HW} for $0\le s\le 1$ we have
\[
 |W(T+s)-\widehat h(T+s)|
 =  |\tW_T(s)+ W(T) - W(T) -H(s)| 
 = |\tW_T(s)-H(s)| \le \delta+|\rho|  \le \delta+r.   
\]
Finally, by \eqref{H1},
\[
\int_T^{T+1} \widehat h'(t)^2 dt
=
\int_0^{1} \widehat H'(s)^2 ds
= |H|_1^2
=  \itwz(1,\delta)^2 +\rho^2 -2\rho  \tW_T(1).
\]
We conclude that
\begin{eqnarray}
  \nonumber
  && \iwz(T+1,r+\delta)^2
  \le |\widehat h|_{T+1}^2
  \\ \label{lower}
  &=& |\widehat h|_T^2+  \int_T^{T+1} \widehat h'(t)^2 dt
  \le \iw(T,r)^2 + \itwz(1,\delta)^2 +r^2 -2\rho  \tW_T(1)
\end{eqnarray}
and turn this relation into the desired bound
\[
     \iw(T,r)^2 
     \ge  \iwz(T+1,r+\delta)^2 - \itwz(1,\delta)^2 - r^2 +2\rho  \tW_T(1).
\]
Notice that $\rho$ and $\tW_T(1)$ are independent and $\E \tW_T(1)=0$. 
By taking expectations we get the desired relation \eqref{lowerE}.
\end{proof}

\subsection{Concentration}

We first notice an almost obvious Lipschitz property of the functionals under consideration.

\begin{prop} 
For any $T,r>0$, any $w\in \CC[0,T], g\in AC[0,T]$ we have
\be \label{lip0}
  \left| I_{w+g}^0(T,r)- I_w^0(T,r) \right|\le |g|_T
\ee
and
\be \label{lip}
  \left| I_{w+g}(T,r)- I_w(T,r) \right|\le |g|_T
\ee 
\end{prop}

It is remarkable that the Lipschitz constant in the right hand side does not depend
on $r$ and $T$.

\begin{proof} Let $h$ be the taut string at which $I_w^0(T,r)$ is attained. Then the function
$\widehat h:= h+g$ satisfies the boundary conditions $\widehat h(0)=(w+g)(0)$,
$\widehat h(T)=(w+g)(T)$ as well as
\[
   ||\widehat h-(w+g)||_T = ||(h+g)-(w+g)||_T = ||h-w||_T \le r.
\]  
Therefore, 
\[
    I_{w+g}^0(T,r) \le |\widehat h|_T \le |h|_T + |g|_T =  I_{w}^0(T,r) + |g|_T. 
\]
By applying the latter inequality to $\tilde w:=w+g$ and $\tilde g:=-g$ in place of
$w$ and $g$ we obtain
\[
    I_{w}^0(T,r) \le  I_{w+g}^0(T,r) + |g|_T, 
\]
and \eqref{lip0} follows. The proof of \eqref{lip} is exactly the same. 
\end{proof}

In the rest of the subsection parameters $T$ and $r$ are fixed, and we drop them from our notation,
thus writing $\iwz$ instead of $\iwz(T,r)$, etc. Let $m^0$ be a median for the random 
variable $\iwz$. The famous concentration inequality for Lipschitz functionals of Gaussian random 
vectors (see \cite[Section 12]{Lif}) asserts that for any $\rho>0$
\begin{eqnarray*}
  \P\left( \iwz\ge m^0+ \rho \right) &\le& \P\left( N\ge \rho \right),
  \\
  \P\left( \iwz\le m^0- \rho \right) &\le& \P\left( N\ge \rho \right),
\end{eqnarray*}
where $N$ is a standard normal random variable.

It follows that
\begin{eqnarray*}
  Var \iwz &=&\inf_{y} \E(\iwz-y)^2 \le \E(\iwz-m^0)^2
  =2\int_0^\infty \rho \, \P(|\iwz-m^0|\ge \rho) \, d\rho
  \\
  &\le&   2\int_0^\infty \rho \, \P(|N|\ge \rho) \, d\rho
  = \E |N|^2 =1.
\end{eqnarray*}
Moreover,
\[
  |\E \iwz-m^0| \le \E |\iwz-m^0| \le \sqrt{\E(\iwz-m^0)^2}
  \le 1 
\]
and 
\[
  \E \iwz \le \sqrt{\E [(\iwz)^2]} = \sqrt{[\E \iwz]^2+ Var \iwz}
  \le \sqrt{[\E \iwz]^2+ 1} \le \E \iwz +1. 
\]
Finally, we infer
\be \label{m0D0}
    m^0-1 \le \sqrt{\E [(\iwz)^2]} \le m^0 +2. 
\ee

We will also need that for any $q>0$
\begin{eqnarray} \nonumber
  \E |\iwz-m^0|^q
  &=&  q \int_0^\infty \rho^{q-1} \, \P(|\iwz-m^0|\ge \rho) \, d\rho
  \\ \label{Imzmean}
  &\le&   q \int_0^\infty \rho^{q-1} \, \P(|N|\ge \rho) \, d\rho
  = \E |N|^q.
\end{eqnarray}
Similarly, for the median $m$ of $\iw$ we obtain
\[
    m-1 \le \sqrt{\E [(\iw)^2]} \le m +2. 
\]
and
\be  \label{Immean} 
  \E |\iw-m|^q  \le \E |N|^q.
\ee

\section{Asymptotics}
\label{s:asymp}
\setcounter{equation}{0}

\subsection{Asymptotics of the second moments and medians}
Recall that
 $D(T,R):=\E \iw(T,r)^2$ and $D^0(T,R):= \E \iwz(T,r)^2$ are the second 
 moments. We prove the following.

\begin{prop} 
There exists a constant $\C\in [0,\infty)$ such that
if $\tfrac{r}{\sqrt{T}}\to 0$, then
\begin{eqnarray} 
   \label{Dtoc2}
   \frac{r^2} {T}\,  D(T,r) \to \C^2,
   \\
   \label{Dtocz2}
   \frac{r^2}{T}    \, D^0(T,r) \to \C^2. 
\end{eqnarray}
\end{prop}

\begin{proof} 
In proving \eqref{Dtoc2}, the following sub-additivity property plays the key role.
For any $r,T_1,T_2>0$ we have
\be \label{add1}
     \iwz(T_1+T_2,r)^2 \le  \iwz(T_1,r)^2 +  I_{{\tW_{T_1}}}^0 (T_2,r)^2, 
\ee
where $\tW_{T_1}(s):=W(T_1+s)-W(T_1)$ is a Wiener process. 
This means that we may approximate $W$ by taut strings with fixed 
ends separately on the intervals $[0,T_1]$ and $[T_1, T_1+T_2]$ 
by gluing them at $T_1$ due to the fixed end condition 
imposed on the first string.

Notice that $\iw(\cdot,r)$ does not possess such a nice subadditivity property.

By taking expectations in \eqref{add1}, we obtain
\be \label{add2}
   D^0(T_1+T_2,r)\le  D^0(T_1,r) + D^0(T_2,r). 
\ee
Since $\iwz(1,\eps)$ is a decreasing random function w.r.t. argument 
$\eps$, the function $D^0(1,\eps)=\E\iwz(1,\eps)^2$ is also decreasing 
in $\eps$. By the scaling argument \eqref{scal0} we observe that for 
any fixed $r>0$
\be \label{scal2}
   D^0(T,r)=D^0\left(1,\tfrac{r}{\sqrt{T}}\right)
\ee    
is an increasing function w.r.t. the argument $T$.

Fix any $T_0>0$. By using monotonicity of $D^0(T,r)$ in $T$ and iterating 
subadditivity \eqref{add2} we obtain
\begin{eqnarray*}
  \limsup_{T\to \infty} \frac{D^0(T,r)}{T} 
  &=&    \limsup_{k\to \infty} \max_{0\le \tau\le T_0} 
     \frac{D^0(kT_0+\tau,r)}{kT_0+\tau}
  \\
  &\le&    \limsup_{k\to \infty}  \frac{D^0((k+1)T_0,r)}{kT_0} 
  \\
  &\le&    \lim_{k\to \infty}  \frac{(k+1) D^0(T_0,r)}{kT_0}
  =  \frac{D^0(T_0,r)}{T_0}.
\end{eqnarray*} 
By optimizing over $T_0$ we find
\[
    \limsup_{T\to \infty} \frac{D^0(T,r)}{T} 
    \le 
    \inf_{T>0} \frac{D^0(T,r)}{T} 
    \le  \liminf_{T\to \infty} \frac{D^0(T,r)}{T}\ . 
\]
It follows that there exists a finite limit
\[
   \lim_{T\to \infty} \frac{D^0(T,r)}{T} =  \inf_{T>0} \frac{D^0(T,r)}{T} := C_r.
\]
By using the scaling \eqref{scal2} we find the limit
\begin{eqnarray*}
\C^2  &:=&   \lim_{\eps\to 0} \eps^2 D^0(1,\eps)
\\
&=&
   r^2 \lim_{T\to \infty} \frac{D^0(T,r)}{T} = C_r \, r^2.
\end{eqnarray*}
Now the relation \eqref{Dtocz2} with varying $r$ follows by 
another application of the scaling argument \eqref{scal2}. 

Now we pass to \eqref{Dtoc2}. For fixed $r>0$ it follows from 
\eqref{comp_simple} and \eqref{Dtocz2} that
\[
    \limsup_{T\to \infty} \frac{D(T,r)}{T} 
    \le
    \lim_{T\to \infty} \frac{D^0(T,r)}{T} = C_r = \C^2 \, r^{-2}.
\]
Conversely, for any fixed $\delta>0$ it follows from \eqref{lowerE} that
\[
    \liminf_{T\to \infty} \frac{D(T,r)}{T} 
    \ge
    \lim_{T\to \infty} \frac{D^0(T,r+\delta)}{T} = \C^2 \, (r+\delta)^{-2}.
\]
By letting $\delta\to 0$ we infer
\[
    \lim_{T\to \infty} \frac{D(T,r)}{T} 
    = \C^2 \, r^{-2}.
\]
which is \eqref{Dtoc2} for fixed $r$. 
The case of varying $r$ in \eqref{Dtoc2} follows by the same scaling arguments 
as above. 
\end{proof}

\begin{remark}
   We will show later in Subsection \ref{ss:olb} that $\C>0$.
\end{remark}

We may complete the convergence of second moments with convergence of medians.

\begin{cor}
Let $ m^0(T,r)$, resp. $m(T,r)$, be a median of $\iwz(T,r)$,
resp. $\iw(T,r)$. If $\tfrac{r}{\sqrt{T}}\to 0$, then
\begin{eqnarray} 
   \label{mtoc}
  \frac{r}{\sqrt{T}} \  m(T,r) \to \C,
   \\
   \label{mtocz}
   \frac{r}{\sqrt{T}} \ m^0(T,r) \to \C. 
\end{eqnarray}
\end{cor}

\begin{proof} Indeed \eqref{m0D0} writes as
\[
   m^0(T,r)-1 \le \sqrt{D^0(T,r)} \le m^0(T,r) +2. 
\]
Therefore, \eqref{mtocz} follows immediately from \eqref{Dtocz2}.
Relation \eqref{mtoc} follows from \eqref{Dtoc2} in the same way.
\end{proof}

\subsection{ $L_q$-convergence}
\label{ss:asym2}

{\bf Proof of Theorem \ref{t:mean}. }
 Let $q>0$. We have to prove that
if $\tfrac{r}{\sqrt{T}}\to 0$, then
\begin{eqnarray} 
   \label{Itocmean}
   \frac{r}{\sqrt{T}} \ \iw(T,r) \tomean \C,
   \\
   \label{Itoczmean}
   \frac{r}{\sqrt{T}}  \ \iwz(T,r) \tomean \C. 
\end{eqnarray}

 In view of \eqref{mtocz} the proof of \eqref{Itoczmean} reduces to
\[
   \frac{r}{\sqrt{T}}   \left( \iwz(T,r) - m^0(T,r) \right) \tomean 0. 
\] 
Indeed by \eqref{Imzmean} we have
\[
   \left( \frac{r}{\sqrt{T}}  \right)^q   \E |\iwz(T,r)-m^0(T,r)|^q  
   \le \left(\frac{r^2}{T} \right)^{q/2}  \E |N|^q \to 0
\]
and \eqref{Itoczmean} follows.

Relation \eqref{Itocmean} follows from \eqref{mtoc} and \eqref{Immean} 
in the same way.
$\Box$

\subsection{ Almost sure convergence}
\label{ss:asym3}

{\bf Proof of Theorem \ref{t:alsur}. }
For any fixed $r>0$, when $T\to\infty$, we must prove
\begin{eqnarray}  
   \label{tocas2}
   \frac{r} { T^{1/2}}\,  \iw(T,r) \toalsur \C,
   \\
   \label{toczas2}
   \frac{r} { T^{1/2}}\, \iwz(T,r) \toalsur \C.   
\end{eqnarray}

Consider first an exponential subsequence $T_k:=a^k$ with arbitrary fixed
$a>1$. By moment estimate \eqref{Immean} and Chebyshev inequality, for any $\eps>0$
we have 
\begin{eqnarray*} 
  \sum_{k=1}^\infty \P\left( T_k^{-1/2} \, |\iw(T_k,r)- m(T_k,r)| >\eps\right)
  &\le& 
  \eps^{-q}  \, \sum_{k=1}^\infty T_k^{-q/2}   \E |\iw(T_k,r)- m(T_k,r)|^q
  \\
   &\le& 
  \eps^{-q}   \E |N|^q  \sum_{k=1}^\infty T_k^{-q/2} <\infty.  
\end{eqnarray*} 
Borel--Cantelli lemma yields 
\[
  \lim_{k\to \infty} T_k^{-1/2}  \left(\iw(T_k,r)- m(T_k,r)\right) =0  \qquad \textrm{a.s.}
\]
Taking the convergence of medians \eqref{mtoc} into account, we obtain
\[
  \lim_{k\to \infty}  \frac{r} {T_k^{1/2}}  \ \iw(T_k,r)  =\C \qquad \textrm{a.s.}
\]
Since the function $\iw(\cdot,r)$ is non-decreasing, for any $T\in [T_k,T_{k+1}]$
we have the chain
\[
  \frac{r \iw(T_k,r)}{(aT_k)^{1/2}}
  = \frac{r \iw(T_k,r)}{T_{k+1}^{1/2}}
  \le \frac{r \iw(T,r)}{T^{1/2}}
  \le \frac{r \iw(T_{k+1},r)}{T_{k}^{1/2}}
  = \frac{r \iw(T_{k+1},r)}{(T_{k+1}/a)^{1/2}}.
\]
It follows that
\[
  a^{-1/2} \C \le \liminf_{T\to \infty}  \frac{r} {T^{1/2}}  \ \iw(T,r)  
  \le \limsup_{T\to \infty}  \frac{r} {T^{1/2}}  \ \iw(T,r) 
  \le   a^{1/2} \C \qquad \textrm{a.s.}
\]
By letting $a\searrow 1$ we obtain \eqref{tocas2}.

The inequality \eqref{comp_simple} yields now the lower bound in \eqref{toczas2},
namely,
\[
  \liminf_{T\to \infty}  \frac{r} {T^{1/2}}  \ \iwz(T,r)  
  \ge
  \lim_{T\to \infty}  \frac{r} {T^{1/2}}  \ \iw(T,r)
  = \C  \qquad \textrm{a.s.}
\]

The proof  of the upper bound in \eqref{toczas2} requires more efforts 
because the monotonicity of $\iwz(\cdot,r)$ is missing. By using \eqref{lower}, 
for any $r>0,\delta>0$ we have
\begin{eqnarray*}
 \limsup_{T\to \infty}  \frac{r^2} {T}\ \iwz(T+1,r+\delta)^2
 &\le&    \limsup_{T\to \infty}  \frac{r^2} {T}  \left( 
 \iw(T,r)^2 + \itwz(1,\delta)^2 +r^2 +2r  |\tW_T(1)|\right)
 \\
 &\le& 
  \C^2 + \limsup_{T\to \infty}  \frac{r^2} {T} \,\itwz(1,\delta)^2  
  +2r   \limsup_{T\to \infty}  \frac{r^2} {T} \, |\tW_T(1)|. 
 \end{eqnarray*}
Now we show that both remaining limits vanish. Indeed, It is well known 
that
\[
    \limsup_{T\to \infty}  \frac{|\tW_T(1)|} {\sqrt{2\ln T}} 
    =  \limsup_{T\to \infty}  \frac{|W(T+1)-W(T)|} {\sqrt{2\ln T}} = 1, 
\]
hence,
\[
    \limsup_{T\to \infty}  \frac{|\tW_T(1)|}{T}  = 0. 
\]
Now write
\[ 
 \limsup_{T\to \infty}  \frac{\itwz(1,\delta)^2} {T} =
  \limsup_{k\to \infty} \sup_{k\le T\le k+1} \frac{\itwz(1,\delta)^2} {k}
  :=  \limsup_{k\to \infty} \frac{V_k(W)^2} {k}\ ,  
\]
where $V_k(W)$ are identically distributed random variables satisfying Lipschitz
condition due to \eqref{lip0}. Let $m$ be the common median of $V_k$. 
By concentration inequality it follows that for any $x>0$
\[
  \P\{ V_k(W)>m+x\} \le \P(N>x) \le \exp\{-x^2/2\}.
\]
Borel--Cantelli lemma yields now that 
\[
   \limsup_{k\to \infty} \frac{V_k(W)} {\sqrt{2\ln k}} \le 1,  
\]
Hence,
\[
   \limsup_{k\to \infty} \frac{V_k(W)^2} {k} =0.
\]
We conclude that
\[
 \limsup_{T\to \infty}  \frac{r^2} {T}\ \iwz(T+1,r+\delta)^2
 \le \C^2
\]
and by letting $\delta\to 0$ we are done with proving upper 
bound in \eqref{toczas2}.
$\Box$

\section{Quantitative estimates and algorithms}
\label{s:estimates}
\setcounter{equation}{0}

In this section we provide several theoretical lower and upper bounds
for $\C$. 

\subsection{Isoperimetric and small deviation bounds}
\label{ss:gkbound}

This subsection closely follows the ideas of Griffin and Kuelbs \cite{GK}.
Let $c>0$. Then for any $\eps>0$ we have
\[
  \P\left(\eps \, \iw(1,\eps)\ge c \right) 
  = \P\left( \iw(1,\eps)\ge c\, \eps^{-1} \right)
  = \P\left( W \not\in \eps U +  c\, \eps^{-1} K  \right) 
\]
where $U:=\{x: ||x||_1\le 1\}$ and $K:=\{h: |h|_1\le 1\}$. According to the 
Gaussian isoperimetric inequality (cf. \cite{Bor, STs}, or e.g. \cite[Section 11]{Lif}),
\[
  \P\left( W\not \in \eps U +  c\, \eps^{-1} K  \right) \le
  1- \Phi\left(  c\, \eps^{-1} + \Phi^{-1}\left(\P(W\in\eps U\right) \right),
\]
where $\Phi(\cdot)$ is the distribution function of the standard normal law.
It is well known that $\Phi^{-1}(p)\sim -\sqrt{2|\ln p|}$, as $p\to 0$.
On the other hand, by the classical small deviation estimate,
following from the Petrovskii formula of the  distribution of $||W||_1$ (cf.
\cite{Ptr} or e.g. \cite[Section 18]{Lif}) 
\[
  \ln \P(W\in\eps U ) =  \ln \P(||W||_1\le\eps) 
  \sim -\frac{\pi^2}{8} \ \eps^{-2} , \qquad \textrm{as } \eps\to 0.
\]
Hence,
\[
  \Phi^{-1}\left(\P(W\in\eps U)\right) 
  \sim  -\frac{\pi}{2} \ \eps^{-1} , \qquad \textrm{as } \eps\to 0.
\]
It follows that 
\[
   \P\left( W\not \in \eps U +  c \, \eps^{-1} K  \right) 
   \le  1- \Phi\left(  c \, \eps^{-1} -\frac{\pi}{2} \, \eps^{-1} (1+o(1)) \right)
   \to 0, \qquad \textrm{as } \eps\to 0,
\]
whenever $c> \frac{\pi}{2}$. Since $\eps\, \iw(1,\eps)\toprob \C$ by \eqref{toc1}, 
we end up with the bound
\[
    \C \le \frac{\pi}{2}\ .
\] 

\subsection{Free knot approximation: constructive approach}
Here we provide a more constructive approach to building strings having 
the right order of energy and properly approximating Wiener process. 
Let $\eps>0$ and let $W(s)$, $0\le s\le 1$, be a Wiener process. Consider 
a sequence of stopping times $\tau_j$ defined by $\tau_0:=0$ and 
\[
   \tau_j:=\inf\{t\ge \tau_{j-1}:\, |W(t)-W(\tau_{j-1})|\ge \eps/2\},\qquad j\ge 1.
\] 
By continuity of $W$ we clearly have
\[ 
    |W(\tau_j)-W(\tau_{j-1})|   = \frac{\eps}{2} \ .
\]
Let $g(\cdot)$ be the linear interpolation of $W(\cdot)$ built upon the knots
$(\tau_j,W(\tau_j))$. We stress that the knots are random, since they depend on 
the process trajectory $W(\cdot)$. This randomness is typical for {\it free 
knot approximation}, cf. \cite{Cr1,Cr2}. 

We have a good approximation of $W$ by $g$ in the uniform norm, since for any
$t\in [\tau_{j-1}, \tau_j]$ it is true that
\begin{eqnarray*}
 |W(t)-W(\tau_{j-1})| &\le& \frac{\eps}{2}\ ,
 \\
 |g(t)- W(\tau_{j-1})| &\le&  |W(\tau_j)-W(\tau_{j-1})| = \frac{\eps}{2}\ , 
\end{eqnarray*}
hence
\be \label{gW}  
   ||g-W||_1\le \eps.
\ee
Let us now evaluate Sobolev norm $|g|_1$. First, we determine the required 
number of knots $N_\eps$ from the condition
\[
   \tau_{N_\eps-1} <1 \le  \tau_{N_\eps}. 
\]
Then
\begin{eqnarray*}
  |g|_1^2 &=& \int_0^1 g'(s)^2 ds 
 \\
 &\le& \sum_{j=1}^{N_\eps}  
     \frac{ \left( W(\tau_j)-W(\tau_{j-1})\right)^2}  { \tau_{j} -\tau_{j-1}} 
  =  \sum_{j=1}^{N_\eps}  \frac{(\eps/2)^2}{\Delta_j} \, 
\end{eqnarray*}
where $\Delta_j:= \tau_{j} -\tau_{j-1}$ are independent random variables identically
distributed with  $(\eps/2)^2\theta$ and 
\[
  \theta:=\inf\{t>0:\ |W(t)|=1\}.
\] 
Therefore, 
\be \label{sumtheta}
     |g|_1^2\le \sum_{j=1}^{N_\eps}  \theta_j^{-1}
\ee
where $\theta_j$ are independent copies of $\theta$. Recall that
$E_1:=\E\theta<\infty$ and $E_2:=\E(\theta^{-1})<\infty$.
By applying the law of large numbers we show that $N_\eps$ has order of
growth $\eps^{-2}=n$, and that, by the same argument, the sum in the right hand side 
of \eqref{sumtheta} also has the same order. Indeed, let $c>0$. Then
\begin{eqnarray*}
    \P\left(N_\eps>c\, \eps^{-2}\right) 
    &=& \P\left(  \sum_{j=1}^{c\,\eps^{-2}} \Delta_j<1\right) 
     = \P\left( (\eps^2/4) \sum_{j=1}^{c\, \eps^{-2}} \theta_j <1\right)
     \\
     &=& \P\left(  \frac{1}{c\,\eps^{-2}} \sum_{j=1}^{c\,\eps^{-2}} \theta_j < \frac{4}{c}\right)
     \to 0,
\end{eqnarray*}
whenever $\tfrac{4}{c}<E_1$. 
Furthermore, for any $v>0$
\[
    \P\left( \sum_{j=1}^{c\,\eps^{-2}} \theta_j^{-1} 
    \ge  \frac{v^2}{\eps^2}\right)
    =  \P\left( \frac{1}{c\,\eps^{-2}}  
    \sum_{j=1}^{c\,\eps^{-2}} \theta_j^{-1} \ge  \frac{v^2}{c}\right)
    \to 0,
\]
whenever $v^2>c\, E_2$.

By \eqref{gW}, we have $\iw(1,\eps)\le |g|_1$. Therefore, by using  
\eqref{sumtheta} and subsequent estimates, we have
\begin{eqnarray*}
   \P\left(\eps \iw(1,\eps) \ge v \right) &\le& \P\left(\eps |g|_1 \ge v \right)
   = \P\left( |g|_1^2 \ge v^2\eps^{-2} \right)
   \\
   &\le& \P\left( \sum_{j=1}^{N_\eps}  \theta_j^{-1} \ge v^2\eps^{-2}   \right)
   \\
   &\le& \P\left(N_\eps>c\, \eps^{-2}\right) 
     + \P\left( \sum_{j=1}^{c\,\eps^{-2}} \theta_j^{-1} \ge  \frac{v^2}{\eps^2}\right).  
\end{eqnarray*}
It follows that
\[
   \P\left(\eps \iw(1,\eps) \ge v \right) \to 0, \qquad \textrm{as } \eps\to 0,
\]
whenever $\tfrac{4}{c}<E_1$ and $v^2>c\, E_2$. By letting $ v^2 \searrow c\,E_2$, 
$c\searrow \tfrac{4}{E_1}$, we obtain
\[
   \P\left(\eps \iw(1,\eps) \ge x \right) \to 0, \qquad \textrm{as } \eps\to 0,
\]
whenever $x>2\sqrt{E_2/E_1}$. Since $\eps\, \iw(1,\eps)\toprob \C$ by \eqref{toc1}, 
it follows that
\be \label{ce1e2}
    \C \le  2\sqrt{E_2/E_1} .
\ee 
It is of interest to calculate the constants $E_1$ and $E_2$ in this bound. By Wald identity,
\[
  1=\E W(\theta)^2 = \E \theta= E_1.
\]
Next, let $M_t:=\sup_{0\le s\le t} |W(s)|$. Then for any $r>0$
\[
  \P(\theta^{-1}\ge r)= \P(\theta\le r^{-1}) = \P( M_{r^{-1}}\ge 1)
  =  \P( r^{-1/2}M_{1}\ge 1) = \P(M_{1}^2 \ge r).
\] 
Therefore, $\theta^{-1}$ and $M_1^2$ are equidistributed.
The distribution of $M_1$ is still inconvenient for calculations. However, it
is convenient to work with  $M_\tau$, where $\tau$ is a standard exponential random variable
independent of $W$. Indeed, by \cite[Formula 1.15.2]{BS} we have for any $a>0$
\[
   \P(M_\tau\ge a)=[\cosh(\sqrt{2}a)]^{-1},
\]
hence, by using \cite[Formula 860.531]{Dw},
\begin{eqnarray*}
 \E M_\tau^2 &=& \int_0^\infty 2\, a\,  \P(M_\tau\ge a) \, da 
 = \int_0^\infty \frac{2a}{\cosh(\sqrt{2}a)}\  da
 \\ 
 &=& \int_0^\infty \frac{x}{\cosh(x)}\  dx \approx 1.832.
\end{eqnarray*}
On the other hand,
\[
   \E M_\tau^2= \int_0^\infty \E M_t^2\  e^{-t}\, dt 
   = \int_0^\infty \E M_1^2 \, t \,  e^{-t} dt
   =\E M_1^2.
\]
We conclude that
\[
  E_2= \E\theta^{-1}= \E M_1^2=\E M_\tau^2 \approx 1.832.
\]
Thus numerical bound from \eqref{ce1e2} becomes $\C\le 2\sqrt{1.832}\approx 2.7$.

\subsection{Oscillation lower bound}
\label{ss:olb}

Fix an arbitrary $x>0$ (to be optimized later on).  Let $n$ be a positive integer 
and let denote $\eps:= x n^{-1/2}$.
Let us split the interval $[0,1]$ into $n$ intervals 
$\Delta_j:=[j/n, (j+1)/n]$ of length $n^{-1}$. Let
\[
   Y_j:= \left( \max_{s\in \Delta_j} W(s) - \min_{t\in \Delta_j} W(t) -2\eps\right)_+
   = \left(W(t_j)-W(s_j)-2\eps\right)_+ 
\]
where $s_j$, $t_j$ are the points where the maximum and the minimum of $W$ are attained.
Notice that by the standard properties of Wiener process (self-similarity, independence 
and stationarity of increments) the variables $Y_j$
are independent and identically distributed with $n^{-1/2} \, Y_x$, where
\[
   Y_x:= \left( \max_{0\le s \le 1} W(s) - \min_{0\le t\le 1} W(t) -2x \right)_+ \ge 0.
\]
Take any function $h\in \CC[0,1]$ such that $||h-W||_1\le \eps$. We have
\begin{eqnarray*}
h(s_j)-h(t_j) &\ge& W(s_j)-\eps - (W(t_j)+\eps) = W(s_j) -(W(t_j)- 2\eps;
\\
|h(s_j)-h(t_j)| &\ge& \left( W(s_j) -W(t_j)- 2\eps \right)_+ =Y_j. 
\end{eqnarray*}
Furthermore, by H\"older inequality,
\begin{eqnarray*}
|h|_1^2 &=& \int_0^1 h'(t)^2 dt =\sum_{j=0}^{n-1} \int_{\Delta_j} h'(t)^2 dt 
\\
&\ge& \sum_{j=0}^{n-1} \int_{s_j}^{t_j} h'(t)^2 dt
\ge \sum_{j=0}^{n-1} \frac{ \left(\int_{s_j}^{t_j} |h'(t)| dt\right)^2}
  {|s_j-t_j|}
\\    
&\ge& \sum_{j=0}^{n-1} \frac{ \left| h(s_j)-h(t_j)\right|^2}
  {|s_j-t_j|} 
  \ge n \sum_{j=0}^{n-1}  Y_j^2 =  \sum_{j=0}^{n-1} \ [Y^{(j)}]^2,
\end{eqnarray*}
where $Y^{(j)}$ are i.i.d. copies of $Y_x$.
It follows that
\[
  \iw(1,\eps)^2 
  \ge  \sum_{j=0}^{n-1}\ [Y^{(j)}]^2 , 
\]
thus
\[
  \eps^2 \iw(1,\eps)^2 
  \ge  \eps^2  \sum_{j=0}^{n-1}\, [Y^{(j)}]^2 =
   x^2 n^{-1}  \sum_{j=0}^{n-1}\ [Y^{(j)}]^2   
\]

Since $\eps^2\, \iw(1,\eps)^2\toprob \C^2$ by \eqref{toc1}, and
by the law of large numbers
$n^{-1}  \sum_{j=0}^{n-1}  \, [Y^{(j)}]^2\toprob \E Y_x^2$,
we infer that
\[
  \C^2\ge x^2 \E Y_x^2 >0.
\]
Let us explore what does it mean numerically. Recall that
$Y_x=(R-2x)_+$ where $R$ is the range of Wiener process on the unit interval of time.
According to \cite[Formula 1.15.4(1)]{BS}, $R$ has the following distribution function,
\[
   \P(R\le y) =1 +4\sum_{k=1}^\infty (-1)^k \, \Erfc\left(\frac{ky}{\sqrt{2}}\right),
\]
where $\Erfc(x):=\tfrac{2}{\sqrt{\pi}}\int_x^\infty e^{-u^2} du$.
It follows that $R$ has a density
\[
  p_R(y) = 4 \sqrt{2/\pi} \sum_{k=1}^\infty (-1)^{k+1}  k^2 \exp\left\{-k^2y^2/2\right\}.
\]
Then
\begin{eqnarray*}
  \E Y_x^2 &=& \int_{2x}^\infty (y-2x)^2 p_R(y) dy
  = 4 \sqrt{2/\pi} \sum_{k=1}^\infty (-1)^{k+1}  k^2  
    \int_{2x}^\infty (y-2x)^2  \exp\left\{-k^2y^2/2\right\} dy
\\
&=&  4 \sqrt{2/\pi} \sum_{k=1}^\infty (-1)^{k+1}  
   \left[ (4k x^2 +\tfrac 1k) \sqrt{2\pi}\ \widehat{\Phi}(2kx)-2x \exp\{-2k^2x^2\}  \right],
\end{eqnarray*}
where $\widehat\Phi(\cdot)$ stands for the tail of standard normal distribution.
The series is so rapidly decreasing in $k$ that mostly the first term $k=1$ is relevant.
By making numerical optimization in $x$ we find the best values near $x\approx 0.5$
where we obtain $x^2 \E Y_x^2\approx 0.145$, thus
\[
  \C\ge x \sqrt{\E Y_x^2} \approx 0.381.
\] 

\section{Markovian pursuit} 
\label{s:pursuit}
\setcounter{equation}{0}

In practice, it is often necessary to build an approximation to the process
adaptively, i.e. in real time, because as parameter (viewed as time) 
advances, we may only know the trajectory of approximated  process 
(Wiener process in our case) up to the current instant. 

In this setting, approximation problem becomes a {\it pursuit problem}.
One may think of a person walking with a dog along one-dimensional path.
Wiener process represents the disordered dog's walk, while the person 
tries to keep the dog  on a leash of given length 
by moving with finite speed and expending minimal energy per unit 
of time.

We will construct an absolutely continuous approximating process $x(t)$ 
such that
\[
   |x(t)- W(t)| \le 1, \qquad t\in \R. 
\]
In view of Markov property of Wiener process, the reasonable strategy is to 
determine the derivative $x'(t)$ as a function of the distance to the target,
\[
  x'(t) := b(x(t)-W(t)) 
\]
where the odd function $b(\cdot)$ defined on $[-1,1]$ explodes to $-\infty$ at $1$
and to  $+\infty$ at $-1$, thus preventing the exit of $x(t)- W(t)$ from the 
corridor $[-1,1]$.
 
In this section we will find an optimal speed function $b(\cdot)$. As a by-product, 
this will give us another upper bound on $\C$.
 
Let $X(t):= x(t)-W(t)$. In our setting, $X(\cdot)$ is a diffusion process satisfying
a simple SDE
\[
   dX= b(X)dt - dW(t).
\]
Recall some important facts about the univariate diffusion, cf. \cite[Chapter IV.11]{B},
\cite[Chapter 2]{BS}.

Let
\[
   B(x):= 2 \int^x b(s) ds.
\]
We stress that $B$ is defined as an indefinite integral, i.e. up to an additive constant.
Let 
\[
     p_0(x) :=  e^{B(x)}.
\]
If
\be \label{entranceboard}
    \int_{-1} \frac{dx}{p_0(x)} = \int^1 \frac{dx}{p_0(x)} =\infty,
\ee
then the boundaries $\pm 1$ belong to the {\it entrance} type and do not belong to the {\it exit} 
type in Feller classification. This means that diffusion remains inside the corridor $[-1,1]$ all 
along its infinite horizon of life.
Moreover, the normalized density
\[
   p(x) := Q^{-1} p_0(x),
\]
where $Q:=\int_{-1}^{1} p_o(x) dx$ is the normalizing factor, is the density of the stationary
distribution of diffusion $X$ considered as a mixing Markov process.
We conclude that at large intervals of time $[0,T]$
\begin{eqnarray*}
  \int_0^T x'(t)^2 dt &=&  \int_0^T b(X(t))^2 dt \sim T \int_{-1}^1 b(x)^2 p(x) dx 
  \\
  &=& T \int_{-1}^1 \left(\tfrac{\ln p}{2}\right)'(x)^2 p(x) dx 
  = \frac T4 \int_{-1}^1 \frac{p'(x)^2}{p(x)}\, dx 
  := \frac{I(p)}{4} \ T,  
  \qquad \textrm{as } T\to\infty,  
\end{eqnarray*}
where, quite unexpectedly, {\it Fisher information} $I(p)$ shows up in the asymptotics.

The next step is to solve the variational problem
\[
   \min\left\{ I(p) \ \big| \int_{-1}^1 p(x) dx = 1\right\}
\]
over the set of even densities concentrated on $[-1,1]$ and satisfying  
\eqref{entranceboard}.

Although the solution is well known (see the references below), we recall 
it here for completeness. For Lagrange variation (with one indefinite 
multiplier) we have, for any smooth function $\delta(\cdot)$ supported by $(-1,1)$
\begin{eqnarray*}
  &&I(p+\delta) - \lambda^2 \int_{-1}^1(p+\delta)(x)dx - \left(I(p)-\lambda^2 \int_{-1}^1 p(x)dx \right)
  \\
  &=& \int_{-1}^1  \left[ \frac{(p'(x)+\delta'(x))^2}{p(x)+\delta(x)}- \frac{p'(x)^2}{p(x)}
  -\lambda^2\delta(x)\right] dx
  \\
  &\sim& \int_{-1}^1  \left[ \frac{2p'(x)\delta'(x)}{p(x)}- \frac{p'(x)^2\delta(x)}{p(x)^2}
  -\lambda^2\delta(x)\right] dx
  \\
  &=& \int_{-1}^1  \left[ -2 \left(\frac{p'}{p}\right)'(x)- \frac{p'(x)^2}{p(x)^2}
  -\lambda^2\right] \delta(x) dx, \qquad \textrm{as } \delta\to 0.
\end{eqnarray*}
We obtain variational equation,
\[
  2 \left(\frac{p'}{p}\right)'(x) + \frac{p'(x)^2}{p(x)^2}+\lambda^2 =0.
\]
By letting $\beta(x):=(\ln p)'(x)= \frac{p'}{p}(x)$, we have
$
   2\beta'+\beta^2+ \lambda^2=0
$
which yields $\tfrac{d\beta}{\beta^2+\lambda^2}=\tfrac{-dx}{2}$ and
\[
  \frac{1}{\lambda} \arctan (\beta/\lambda) = c- \frac{x}{2}.
\]
Since by symmetry $p'(0)=0$, we have $\beta(0)=0$, thus $c=0$ and
\[
  \frac{1}{\lambda} \arctan (\beta/\lambda) = - \frac{x}{2},
\]
or, equivalently,
\[
   \beta(x) = -\lambda \tan(\lambda x/2).
\]
Furthermore, since $p(\cdot)$ should vanish on the boundary $\pm 1$, $\beta$ should explode, i.e.
$\beta(\pm 1)= \mp\infty$, we obtain $\lambda=\pi$. Hence,
\[
   \beta(x) = - \pi \tan(\pi x/2).
\]
Next,
\begin{eqnarray*}
  \ln p(x) &=& \int (\ln p)'(x) dx = \int \beta(x) dx
  \\
  &=& - \pi \int \tan(\pi x/2) dx = c + 2 \ln\cos(\pi x/2).
\end{eqnarray*}
Therefore, the density of the optimal invariant measure is
\[
  p(x)= c_1 \, \cos^2(\pi x/2).
\]
Since 
\[
  \int_{-1}^1 \cos^2(\pi x/2)\, dx =  \frac 12 \int_{-1}^1 \left(1+ \cos(\pi x)\right) dx =1,  
\]
we have $c_1=1$, thus
\[
  p(x)= \cos^2(\pi x/2).
\]
This distribution, as a minimizer of Fisher information on an interval, can be also found
in \cite{Hub,Lev}, \cite[p.63]{Shev}.  Luckily for us,  this $p(\cdot)$ satisfies
\[
  \int^1 \frac{dx}{p(x)} = \int_{-1} \frac{dx}{p(x)} =\infty,
\]
so that \eqref{entranceboard} is satisfied, and we really have entrance boards
for the optimal regime.  

It remains now to calculate the optimal Fisher information,
\begin{eqnarray*}
I(p) &=& \int_{-1}^1 \frac{p'(x)^2}{p(x)} \ dx = \int_{-1}^1 b(x)^2 p(x)\  dx
\\
 &=&  \pi^2 \int_{-1}^1  \tan^2(\pi x/2)\ \cos^2(\pi x/2) dx   
 =  \pi^2 \int_{-1}^1  \sin^2(\pi x/2) dx 
 \\
 &=&\pi^2\ .
\end{eqnarray*}
The optimum is attained at the speed strategy
\[
   b(x) =\frac{\beta(x)}{2} = - \frac{\pi}{2}\ \tan(\pi x/2);
\]
see an example of its implementation in Fig. \ref{fig:mp}. 

{\it We conjecture that the provided algorithm is the optimal one 
in the entire class of adaptive algorithms.}  

\begin{center}
\begin{figure} [ht]

\begin{center}
\includegraphics[height=1.7in,width= 3.5in]{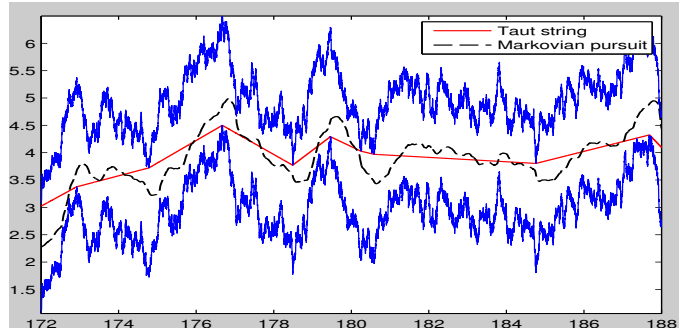}
\end{center}

\caption{Optimal Markovian pursuit accompanying Wiener process}
\label{fig:mp}
\end{figure}
\end{center}


As a by-product we get a bound for non-Markovian asymptotic 
bound,
\[
   \C \le \frac{I(p)^{1/2}}{2} = \frac{\pi}{2}\ ,
\] 
surprisingly the same as the bound obtained in Subsection
\ref{ss:gkbound} by completely different method.

\section{Some simulation results}
\label{s:simulation}
\setcounter{equation}{0}


\begin{center}
\begin{figure} [ht]

\begin{center}
\includegraphics[height=2in,width= 4in]{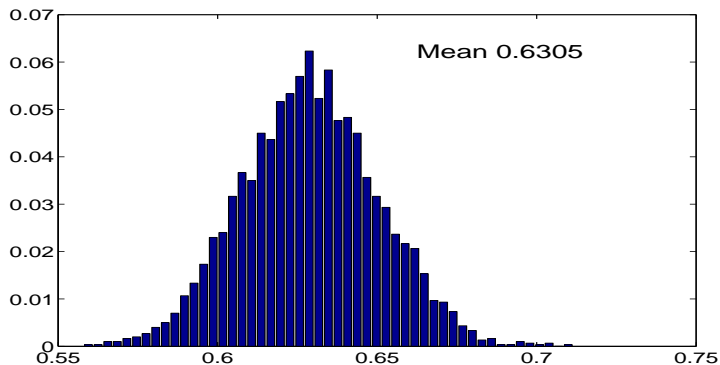}
\end{center}

\caption{Histogram for taut string energy constant $\C$.}
\label{fig:hts}
\end{figure}
\end{center}


We simulate a path of Wiener process $W(t)$ on the interval $\left[0,T\right]$ with $N+1$ uniformly distributed knots, i.e. $\left(iT/N,W(iT/N)\right)$, $i=0,1,...,N$. For each simulated path we consider the tube of radius $1$ and compute the discrete taut string with fixed end and Markovian pursuit functions.
\par Let us describe the computation of the discrete Markovian pursuit. The stochastic differential equation
\begin{equation*}
     h'(t):=-\frac{\pi}{2} \tan \left(\frac{\pi}{2}\left(h(t)-W(t)\right)\right).
\end{equation*}
is discretized with a backward finite difference method which result in the equations
\begin{equation} \label{1}
    \frac{h\left(\frac{iT}{N}\right)-h\left(\frac{(i-1)T}{N}\right)}{\frac{T}{N}}
    =-\frac{\pi}{2}\tan\left(\frac{\pi}{2}\left(h\left(\frac{(i-1)T}{N}\right)
    -W\left(\frac{(i-1)T}{N}\right)\right)\right)
\end{equation}
for $i=1,...,N$ together with the initial condition $h(0)=0$. If $h(iT/N)$ is 
close to the boundaries $W(iT/N)\pm 1$ we might experience numerical 
instability at subsequent time points due to $\tan(\frac{\pi x}{2})
\rightarrow\pm\infty$ when $x\rightarrow\pm 1^{\mp}$. To avoid this, we constrain 
$h(iT/N)$ to the interval
\begin{displaymath}
    \left[W\left(\frac{iT}{N}\right)-0.99,W\left(\frac{iT}{N}\right)+0.99\right].
\end{displaymath}
If the computed value of $h(iT/N)$, given by (\ref{1}), is outside this interval 
we set $h(iT/N)$ equal to the nearest endpoint of the interval.
\par From the discrete taut string and Markovian pursuit functions, piecewise linear functions $h(t)$ on $\left[0,T\right]$ are constructed by joining consecutive pair of the computed knots
\begin{displaymath}
    \left(iT/N,h(iT/N)\right), i=0,1,...,N
\end{displaymath}
with line segments. The (square root of) energy $|h|_T$ can then be computed according to
\begin{equation*}
      |h|_T:=\frac{1}{T^{1/2}} \left(\int_0^T h'(t)^2 dt \right)^{1/2}
      = \frac{N^{1/2}}{T} \left(\sum^{N}_{i=1}\left(h\left(\frac{iT}{N}\right)
      -h\left(\frac{(i-1)T}{N}\right)\right)^{2} \right)^{1/2}.
\end{equation*}
\par We simulated 3000 independent paths of the Wiener process on the interval $\left[0,T\right]$ with $T=1000$ and $N=1000000$. For this sample, the mean of $|h|_T$ for the taut string with fixed end was approximately $0.63$, see Fig. \ref{fig:hts} for the corresponding histogram. Using the same sample of paths, the mean of $|h|_T$ for the Markovian pursuit was approximately $1.62$, see Fig. \ref{fig:hmp}, which is reasonably close to the theoretical constant $\frac{\pi}{2}\approx 1.57$ when $T\rightarrow \infty$.

\begin{center}
\begin{figure} [ht]

\begin{center}
\includegraphics[height=2in,width=4in]{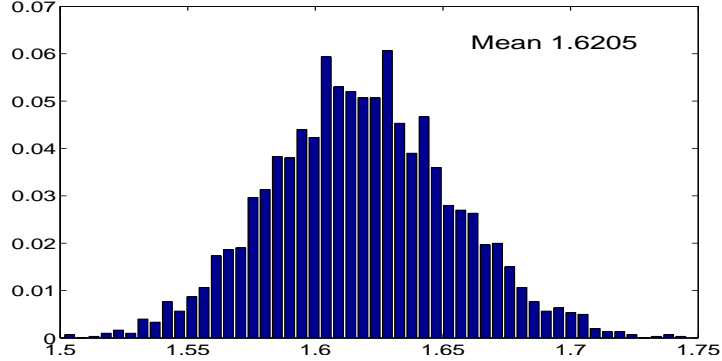}
\end{center}

\caption{Histogram for Markovian pursuit energy constant.}
\label{fig:hmp}
\end{figure}
\end{center}


\section{Some related problems}
\label{s:misc}
\setcounter{equation}{0}

By different reasons, taut strings and similar objects already appeared, 
sometimes implicitly, in probabilistic problems. Therefore, it seems 
reasonable to place our results into a historical context.  

\subsection{Strassen's functional law of the iterated logarithm}

Strassen's functional law of the iterated logarithm \cite{Str,Lif} claims
that
\[
   \limsup_{T\to \infty} \inf_{|h|_1\le 1} 
   \left\| \frac{W(\cdot T)}{\sqrt{2T\ln\ln T}} -h \right\|_1
   =0
   \qquad \textrm{a.s.}
\]
Grill and Talagrand \cite{Gr2,Tal} independently established the optimal convergence
rate in this law by proving that for some finite positive constants $c_1$,
$c_2$ it is true that
\[  c_1  < 
   \limsup_{T\to \infty} \ (\ln\ln T)^{2/3} \inf_{|h|_1\le 1} 
   \left\| \frac{W(\cdot T)}{\sqrt{2T\ln\ln T}} -h \right\|_1
   < c_2
   \qquad \textrm{a.s.}
\]
Due to scaling properties of the function $\iw(\cdot,\cdot)$, the latter statement
just means that
\begin{eqnarray*}
&&\limsup_{T\to\infty} \frac{\iw(T,c_1(2T)^{1/2}(\ln\ln T)^{-1/6})}{(2\ln\ln T)^{1/2}} >1,
\\
&&\limsup_{T\to\infty} \frac{\iw(T,c_2(2T)^{1/2}(\ln\ln T)^{-1/6})}{(2\ln\ln T)^{1/2}} <1.
\end{eqnarray*}
Grill \cite{Gr1} and Griffin and Kuelbs \cite{GK} showed a similar $\liminf$ result
asserting that for some $c_3>0$ and any $c_4>\tfrac{\pi}{8}$ it is true that
\[  c_3  < 
   \liminf_{T\to \infty} \ (\ln\ln T) \inf_{|h|_1\le 1} 
   \left\| \frac{W(\cdot T)}{\sqrt{2T\ln\ln T}} -h \right\|_1
   < c_4
   \qquad \textrm{a.s.}
\]
which means, in our notations, that
\begin{eqnarray*}
&&\liminf_{T\to\infty} \frac{\iw(T,c_3(2T)^{1/2}(\ln\ln T)^{-1/2})}{(2\ln\ln T)^{1/2}} >1,
\\
&&\liminf_{T\to\infty} \frac{\iw(T,c_4(2T)^{1/2}
(\ln\ln T)^{-1/2})}{(2\ln\ln T)^{1/2}} <1.
\end{eqnarray*}

We may conclude that the tubes relevant to the Strassen law are much larger 
than ours. Accordingly, the respective minimal energy is much lower.

\subsection{$L_1$-optimizers: lazy functions}

By its definition, the taut string is the unique minimizer of
$\int_0^T h'(t)^2 dt$ among the functions whose graphs pass through
the corresponding tube. It is well known, however, that when both 
endpoints are fixed, the taut string also is a minimizer for any 
functional $\int_0^T \varphi(h'(t)) dt$ whenever $\varphi(\cdot)$ 
is a convex function. In the recent literature, much attention 
was payed to the case $\varphi(x)=|x|$, i.e. to the minimization 
of variation
\[
   \V(h):= \int_0^T |h'(t)| dt,
\]  
see \cite{LoMil,Mil}. Notice that since $|\cdot|$ is not a {\it strictly} 
convex function, the corresponding variational problem typically has 
{\it many} solutions. Moreover, since the variation is well defined not 
only on 
absolutely continuous functions, the natural functional domain for 
optimization becomes wider. In \cite{Mil} another minimizer of $\V(h)$ is 
described in detail, a so called "lazy function". When possible, 
this function remains constant; otherwise, it follows the boundary of 
the tube. Notice that lazy function need not be absolutely continuous;
it only has a bounded variation.

For the case when the tube is constructed around a sample path 
of Wiener process, \cite{Mil} suggests a description of lazy function
as an inverse to appropriate subordinator. Although the taut string 
and lazy function both solve the same variational problem, 
the relations between them are yet to be clarified.

\subsection{A related discrete applied problem}

We describe in this section an interesting discrete applied problem
coming from information transmission that turns out to be related
with discrete taut string construction. This problem actually was
an initial motivation for our research.
 
Consider the following information transmission unit represented on
Fig. \ref{fig:channel}.


{\unitlength=0.4mm

\begin{center}
\begin{figure} [ht]

\begin{center}
\begin{picture}(210, 70)


\put(0,30){\vector(1,0){10}}

\put(10,17){\line(0,1){25}}
\put(50,17){\line(0,1){25}}
\put(10,17){\line(1,0){40}}
\put(10,42){\line(1,0){40}}

\put(14,35){\tiny Entrance}
\put(14,27){\tiny information}
\put(14,19){\tiny flow}

\put(50,30){\vector(1,0){25}}
\put(60,17){{$\scriptstyle S$}}


\put(75,30){\circle*{3}}
\put(75,30){\vector(1,0){75}}

\put(150,17){\line(0,1){25}}
\put(192,17){\line(0,1){25}}
\put(150,17){\line(1,0){42}}
\put(150,42){\line(1,0){42}}

\put(154,33){\tiny Transmission}
\put(154,25){\tiny Channel}

\put(192,30){\vector(1,0){20}}
\put(200,17){{$\scriptstyle C$}}


\put(75,30){\vector(1,1){22}}

\put(97,52){\line(0,1){20}}
\put(130,52){\line(0,1){20}}
\put(97,52){\line(1,0){33}}
\put(97,72){\line(1,0){33}}

\put(100,62){\tiny Buffer}
\put(100,57){\tiny of size B}

\put(130,52){\vector(2,-1){20}}


\put(75,30){\vector(1,-1){30}}
\put(110,0){\tiny Loss $L$}
\end{picture}
\end{center}

\caption{Information transmission unit.}
\label{fig:channel}
\end{figure}
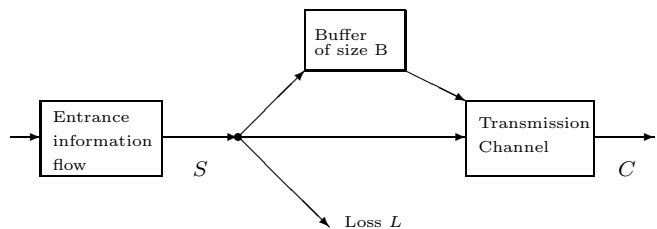
\end{center}

}


We have the discrete time count: $j=1,2,3,\dots$. At each time
$j$ an amount of information $S_j$ enters the system and should
be transmitted through a channel. The channel's {\it transmission capacity}
$C_j$ varies upon the time (for example, the channel may be shared
with other tasks external to our information flow). We are interested in
the situation when the channel capacity is insufficient for transmission, 
i.e. $S_j\ge C_j$. We may place a part of the excessive information 
into a {\it buffer} of given size $B$ and drop (loose) the remaining part.
Let $L_j$ denote the loss size. This variable remains under our 
partial control, yet within buffer size limitations.
Let $B_j$ denote the amount of information stocked in the buffer. One
necessarily has 
\be \label{BjB}
     0\le B_j\le B. 
\ee
Given $\varphi: [0,1] \mapsto \R_+$ -- an increasing convex {\it penalty 
function}, define the {\it penalty functional} by the formula
\[
   F := \sum_{j=1}^n \varphi\left(\tfrac{L_j}{S_j}\right)\, S_j.
\] 
Given $(S_j),(C_j)$, and $B$, we are interested to minimize $F$ by controlling
$L_j$. It is important to notice that eventual non-linearity of $\varphi(\cdot)$
is a natural feature because a small loss of information, e.g. of a graphical one,
is more likely to be repaired by interpolation methods than a large loss.

The process of system work may be analyzed through the buffer balance equation.
We clearly have

\[
   B_j = B_{j-1} +\left(S_j-C_j-L_j \right).
\]
Therefore,
\[
  B_k = \sum_{j=1}^k \left(S_j-C_j\right) - \sum_{j=1}^k  L_j.
\]
Now the buffer bounds \eqref{BjB} mean that
\[
   \sum_{j=1}^k \left(S_j-C_j \right) -B 
   \le \sum_{j=1}^k  L_j 
   \le \sum_{j=1}^k \left(S_j-C_j\right).
\]
In other words, the accumulated loss curve $\sum_{j=1}^k  L_j$ must go within a (random) 
band of fixed width $B$, see Fig. \ref{fig:bb}. Note that on the picture we use the
{\it operational time}, i.e. the accumulated entrance flow  $\sum_{j=1}^k  S_j$ instead of
the usual time $j$.

Therefore the minimum
\[
  F=F(L) = \sum_{j=1}^n \varphi \left(\tfrac{L_j}{S_j}\right)\, S_j
  = \int_0^{S} \varphi(L'(s)) ds \searrow min 
\]
where $S:=\sum_{j=1}^n S_j$,
is attained at the corresponding taut string.

The greedy FIFO strategy ("first in, first out") which consists in 
keeping the buffer full all the time corresponds to the accumulated loss graph
going along the lower border of the admissible corridor. It is usually non-optimal 
at all.

Assuming that information excess  $S_j-C_j$ is a sequence of identically distributed
random variables, we arrive to the problem of construction
of taut string accompanying sums of i.i.d. random variables with positive drift.


{\unitlength=0.4mm

\begin{center}
\begin{figure} [ht]

\vskip 1cm

\begin{center}

\begin{picture}(260, 130)

\put(0,0){\vector(1,0){240}}
\put(0,0){\vector(0,1){140}}

{\thicklines 
\put(0,0){\line(2,3){40}}
\put(40,60){\line(3,1){60}}
\put(100,80){\line(1,2){20}}
\put(120,120){\line(1,0){40}}
\put(160,120){\line(4,1){50}}

\put(0,0){\line(2,1){40}}
\put(40,20){\line(3,1){60}}
\put(100,40){\line(1,2){20}}
\put(120,80){\line(1,0){40}}
\put(160,80){\line(4,1){50}}

\put(0,0){\line(1,1){40}}
\put(40,40){\line(1,0){60}}
\put(100,40){\line(1,3){20}}
\put(120,100){\line(4,1){40}}
\put(160,110){\line(5,2){50}}
} 

\put(40,0){\line(0,1){60}}
\put(100,0){\line(0,1){80}}
\put(120,0){\line(0,1){120}}
\put(160,0){\line(0,1){120}}
\put(210,0){\line(0,1){132}}

\put(153,140){\tiny Accumulated information excess $\sum (S_j-C_j)$}
\put(213,90){\tiny FIFO strategy (full buffer)}

\put(213,120){\tiny Accumulated loss $\sum L_j$}

\put(203,-10){\tiny Accumulated entrance flow $\sum S_j$}

\put(15,-8){\tiny $\scriptstyle S_1$ }
\put(65,-8){\tiny $\scriptstyle S_2$ }
\put(105,-8){\tiny  $\scriptstyle S_3$ }
\put(135,-8){\tiny $\scriptstyle S_4$ }
\put(180,-8){\tiny $\scriptstyle S_5$ }

\put(0,0){\vector(1,0){40}}
\put(40,0){\vector(-1,0){40}}

\put(40,0){\vector(1,0){60}}
\put(100,0){\vector(-1,0){60}}

\put(100,0){\vector(1,0){20}}
\put(120,0){\vector(-1,0){20}}

\put(120,0){\vector(1,0){40}}
\put(1160,0){\vector(-1,0){40}}

\put(160,0){\vector(1,0){50}}
\put(210,0){\vector(-1,0){50}}

\put(136,80){\vector(0,1){40}}
\put(136,120){\vector(0,-1){40}}
\put(138,90){\tiny $\scriptstyle B$ }

\put(120,100){\vector(0,1){20}}
\put(120,120){\vector(0,-1){20}}
\put(122,110){\tiny ${\scriptstyle B_j}$ }
\end{picture}
\end{center}

\vskip 1 cm

\caption{Transmission unit  work graph.}
\label{fig:bb}
\end{figure}
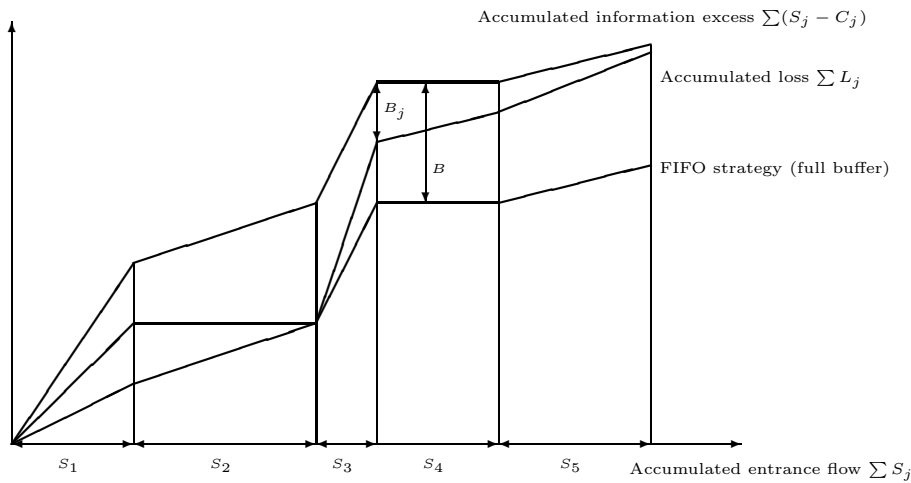
\end{center}

}


\section{Final remarks}
\label{s:concl}

After energy evaluation for the taut strings accompanying Wiener process,
many similar questions arise. 

Within the same framework, it would be natural 
to study more general functionals of taut string by replacing energy with the functionals
$\int_0^T \varphi(h'(t))\, dt$ with more or less general convex function 
$\varphi$. Since in the long run the derivative of accompanying taut string 
seems to be close to an ergodic stationary process, characterized by
its invariant distribution, say $\mu$, it is natural to expect that an ergodic
theorem holds in the form
\[
  \frac 1T \int_0^T \varphi(h'(t))\, dt \toalsur \int_\R \varphi(x)\, \mu(dx) 
\]  
thus extending our Theorem \ref{t:alsur}.

It is natural to explore the energy and similar characteristics of the taut strings
accompanying other processes. The fractional Brownian motion is the first
obvious candidate, but in general, the class of processes with stationary increments
including non-Gaussian L\'evy processes seems to be a natural framework for
this extension. Notice that the energy we handled here has a special relation to
Wiener process, because it coincides with the squared norm of the corresponding
reproducing kernel. This makes our proofs easier but we hope that handling
energy for other processes is still possible.

One can also modify the form of the tube that defines required closeness between
the string and the process. For example if we measure the distance between the string 
and the process in $L_2$-norm instead of the uniform one, then all calculations
become explicit, and the analogue of constant $\C$ may be calculated precisely.
This will be a subject of forthcoming publication. 
\bigskip

{\bf Acknowledgement.}\ The authors are much indebted to Professor Natan
Kruglyak for providing strong motivation for this research and for 
useful discussions. 
They are also grateful to Z.~Kabluchko and E.~Schertzer for enlightening 
comments.
   
The first named author work was supported by grants RFBR 13-01-00172 
and SPbSU 6.38.672.2013.


\end{document}